\documentclass[a4paper]{amsart} 
\usepackage{graphicx}
\usepackage{color}
\usepackage{hyperref}
\usepackage{amssymb,amsmath} 
\usepackage[arrow,curve,matrix,tips]{xy}
\usepackage{enumerate}
\usepackage{esint}
\usepackage[arrow,curve,matrix,tips]{xy}
\usepackage{setspace}
\makeatletter

\newtheorem{theorem}{Theorem}[section]
\newtheorem{lemma}[theorem]{Lemma}

\newtheorem{proposition}[theorem]{Proposition}
\newtheorem{corollary}[theorem]{Corollary}

\theoremstyle{definition} 
\newtheorem{definition}[theorem]{Definition}
\newtheorem{example}[theorem]{Example}

\newtheorem{notation}[theorem]{Notation}

\newtheorem{remark}[theorem]{Remark}

\numberwithin{equation}{section} 
\makeatletter
\def\moverlay{\mathpalette\mov@rlay}
\def\mov@rlay#1#2{\leavevmode\vtop{%
   \baselineskip\z@skip \lineskiplimit-\maxdimen
   \ialign{\hfil$\m@th#1##$\hfil\cr#2\crcr}}}
\newcommand{\charfusion}[3][\mathord]{
    #1{\ifx#1\mathop\vphantom{#2}\fi
        \mathpalette\mov@rlay{#2\cr#3}
      }
    \ifx#1\mathop\expandafter\displaylimits\fi}
\makeatother

\begin{document} 
\title{Quasi-ordered Rings}

\author{Simon M\"uller}

\begin{center} {\small{\today}} \end{center}

\begin{abstract}
A quasi-order is a binary, reflexive and transitive relation. In the Journal of Pure and Applied Algebra 45 (1987), S.M. Fakhruddin introduced the notion of (totally) quasi-ordered fields and showed that each such 
field is either an ordered field or else a valued field. Hence, quasi-ordered fields are very well suited to treat ordered and valued fields simultaneously. \\
In this note, we will prove that the same dichotomy holds for commutative rings with $1$ as well. For that purpose we first develop an appropriate notion of (totally) quasi-ordered rings. Our proof of the dichotomy 
then exploits Fakhruddin's result that was mentioned above. \\
Quasi-ordered Rings are a very interesting class on its own. Their investigation is continued in \cite{Kuhl}, where Salma Kuhlmann and the author of this note develop a notion of compatibility between quasi-orders
and valuations, and establish a Baer-Krull Theorem.  
\end{abstract}

\maketitle

\section{Introduction}
\noindent
Let $S$ be a set and $\preceq$ a binary, reflexive, and transitive relation on $S.$ Then $\preceq$ defines an equivalence relation $\sim \: := \: \sim_{\preceq}$ on $S$ by declaring 
$a \sim b :\Leftrightarrow a \preceq b \textrm{ and } b \preceq a.$ Here $\sim$ always denotes this equivalence relation and we write $a \prec b$ for $a \preceq b$ and $b \npreceq a.$ \\
In his note \cite{Fakh}, Fakhruddin gave the following definition of quasi-ordered fields:

\begin{definition} \label{qofield}
Let $K$ be a field and $\preceq$ a binary, reflexive, transitive and total relation on $K.$ Then $(K,\preceq)$ is called a \textbf{quasi-ordered field} if $\forall x,y,z \in K:$
\begin{itemize}
\item[(Q1)] $x \sim 0 \Rightarrow x = 0,$
\item[(Q2)] $x \preceq y, \; 0 \preceq z \Rightarrow xz \preceq yz,$
\item[(Q3)] $x \preceq y, \; z \nsim y \Rightarrow x+z \preceq y+z.$
\end{itemize}
\end{definition}

\noindent
Quasi-ordered fields unify the classes of ordered and valued fields, as Fakhruddin's main theorem states (see \cite[Theorem 2.1]{Fakh}):

\begin{theorem} \label{qofielddicho} A quasi-ordered field $(K,\preceq)$ is either an ordered field or else a valued field $(K,v)$ such that $x \preceq y$ if and only if $v(y) \leq v(x)$ for all $x,y \in K.$
\end{theorem} 

\noindent
The aim of this note is to establish the same result for quasi-ordered rings, see Theorem \ref{dicho}, respectively Theorem \ref{main}. In section two we briefly recall valued and ordered rings. The object of the 
third section is to introduce ournotion of quasi-ordered rings, and to show that for such rings the equivalence class of $0$ with respect to $\sim,$ denoted by $E_0,$ is a prime ideal. In the fourth and final 
section, we prove the dichotomy in two steps. At first we show that the quasi-order on $R$ can be extended to a quasi-order on the quotient field $\textrm{Quot}(R/E_0),$ see Proposition \ref{quot}, and apply 
Theorem \ref{qofielddicho}. Afterwards we carry out how this gives rise to a suitable order, respectively a suitable valuation, on $R.$

\section{Ordered and valued rings}

\noindent
For the rest of this section, let $R$ always denote a commutative ring with $1.$ \vspace{2mm}

\noindent
In real algebra it is common to regard order relations as unary relations, i.e. as subsets of the algebraic structure under consideration, and to declare precisely the elements of this subset as the non-negative 
elements. In this sense, orders on rings are identified with positive cones (see for instance \cite[p. 29]{Marsh}):

\begin{definition} \label{cone} A \textbf{positive cone} of $R$ is a subset $P \subset R$ such that the following conditions are satisfied:
\begin{itemize}
\item[(P0)] $P \cup -P = R$ (totality),
\item[(P1)] $\mathfrak{p}:= P \cap -P$ is a prime ideal of $R,$ called the \textbf{support} of $R,$
\item[(P2)] $P \cdot P \subseteq P,$ 
\item[(P3)] $P + P \subseteq P.$
\end{itemize}
\end{definition}

\noindent
The correspondence between positive cones $P \subset R$ and orders $\leq \: \subset R^2$ is then given by $x \leq y \Leftrightarrow y-x \in P.$ However, as the next lemma shows, such unary descriptions of 
order relations can only work if the compatibility of the order with addition is not restricted.

\begin{lemma} Let $T$ be a subset of $R.$ For $x,y \in R$ define $x \leq y :\Leftrightarrow y-x \in T.$ Then $\leq$ already satisfies $\forall x,y,z \in R: x \leq y \Rightarrow x+z \leq y+z.$
\end{lemma}
\begin{proof}
Suppose $x \leq y,$ so $y-x \in T.$ Then also $(y+z)-(x+z) \in T$ for any $z \in R.$ Consequently $x+z \leq y+z.$
\end{proof}

\noindent
Thus, in order to deal with axioms like (Q3), it is necessary to work with orders (binary) instead of positive cones (unary). This motivates: 

\begin{definition} \label{oring} Let $\leq$ be a binary, reflexive, transitive and total relation on $R.$ Then $(R,\leq)$ is called an \textbf{ordered ring} if $\forall x,y,z \in R:$ 
\begin{itemize}
 \item[(O1)] $0 < 1,$
 \item[(O2)] $xy \leq 0 \Rightarrow x \leq 0 \lor y \leq 0,$
 \item[(O3)] $x \leq y, \; 0 \leq z \Rightarrow xz \leq yz,$
 \item[(O4)] $x \leq y \Rightarrow x+z \leq y+z.$
 \end{itemize}
\end{definition}

\noindent
We claim that the Definitions \ref{cone} and \ref{oring} are equivalent. In Lemma \ref{def} and Corollary \ref{def2} below, we only prove that the axioms (O1) and (O2) correspond to (P1), hereby pretending that the 
rest of the proof, which is straightforward, was already carried out. As mentioned in the introduction, let $\sim \: := \: \sim_{\leq}$ be the equivalence relation given by $x \sim y \Leftrightarrow x \leq y$ and 
$y \leq x,$ and let $E_0$ be the equivalence class of $0.$ Note that for ordered rings the sets $E_0$ and $\mathfrak{p} = P \cap -P$ coincide.

\begin{lemma} \label{def} Let $(R,\leq)$ be an ordered ring, except that axiom $\mathrm{(O2)}$ is omitted. Then $\mathrm{(O2)}$ holds if and only if $\mathrm{(O2')}$ $xy \sim 0 \Rightarrow x \sim 0 \lor y \sim 0$ 
holds.
\end{lemma}
\begin{proof} We first prove the only if part. Let $xy \leq 0.$ Assume that $0<x$ and $0<y.$ Then (O3) yields $0 \leq xy$ and therewith $xy \sim 0.$ From the assumption (O2') follows $x \sim 0$ or $y \sim 0,$ a 
contradiction. Hence $x \leq 0$ or $y \leq 0.$ Now suppose that (O2) holds. Show that $x \nsim 0$ and $y \nsim 0$ implies $xy \nsim 0.$ Assume without loss of generality that $0<x$ and $0<y$ (if for instance $x<0,$ 
then $0<-x$ and we continue with $-x$ instead of $x$). Then (O2) yields $0<xy.$ Therefore $xy \nsim 0.$ 
\end{proof}

\noindent
Obviously axiom (P1) implies both (O1) and (O2'), and so by Lemma \ref{def} it also implies (O2). The following corollary states that the converse is also true, i.e. from the axioms of an ordered ring follows
that $\mathfrak{p}$ is a prime ideal.

\begin{corollary} \label{def2} If $(R, \leq)$ is an ordered ring, then $E_0 = \{x \in R: x \sim 0\},$ which is equal to $\mathfrak{p},$ is a prime ideal of $R.$
\begin{proof} Evidently $0 \in E_0$ by reflexivity. Now let $x,y \in E_0.$ From $0 \leq x$ and $0 \leq y$ follows $0 \leq y \leq x+y.$ Transitivity yields $0 \leq x+y.$ Analogously we obtain $x+y \leq 0$ and hence 
$x+y \in E_0.$ Next, let $x \in E_0$ and $y \in R.$ The case $0 \leq y$ is clear. If $y \leq 0,$ then $0 \leq -y.$ Therefore $0 \leq -xy$ and $-xy \leq 0.$ By (O4) we obtain that $xy \in E_0.$ Thus, $E_0$ is an 
ideal. By Lemma \ref{def} and axiom (O1), $E_0$ is a prime ideal of $R.$
\end{proof}
\end{corollary}

\noindent
We conclude this section by recalling ring valuations.

\begin{definition} (see \cite[VI. 3.1]{Bo}) Let $(\Gamma,+,\leq)$ be an ordered abelian group and $\infty$ a symbol such that $\gamma < \infty$ and $\infty = \infty + \infty = \gamma + \infty = \infty + \gamma$ for
all $\gamma \in \Gamma.$ \\
A map $v: R \to \Gamma \cup \{\infty\}$ is called a \textbf{valuation} on $R$ if $\forall x,y \in R:$
\begin{itemize}
 \item[(V1)] $v(0) = \infty,$ 
 \item[(V2)] $v(1) = 0,$
 \item[(V3)] $v(xy) = v(x) + v(y),$
 \item[(V4)] $v(x+y) \geq \min\{v(x),v(y)\}.$
\end{itemize}
We always assume that $\Gamma$ is the group generated by $\{v(x): x \notin v^{-1}(\infty)\}$ and call it the \textbf{value group} of $R.$ The set $\mathfrak{q}:= \mathrm{supp}(v) := v^{-1}(\infty)$ is called the 
\textbf{support} of $v.$
\end{definition}

\noindent
Note that ring valuations are in general not surjective. This is because $v(R-\mathfrak{q})$ is not necessarily closed under additive inverses. However, if $R$ is a field and $x \in R$ with 
$v(x) = \gamma \in \Gamma,$ then it is easy to see that $v(x^{-1}) = -\gamma.$

\begin{lemma} If $v$ is a valuation on $R,$ then its support $\mathfrak{q}$ is a prime ideal of $R.$
\begin{proof} $0 \in \mathfrak{q}$ and $1 \notin \mathfrak{q}$ by the axioms (V1) and (V2), respectively. Now let $x,y \in \mathfrak{q}.$ Via (V4) we obtain $v(x+y) \geq \min\{v(x),v(y)\} = \infty,$ so 
$x+y \in \mathfrak{q}.$ Next, let $x \in \mathfrak{q}$ and $r \in R.$ (V3) implies $v(rx) = v(r)+v(x) = \infty,$ whereby $xr \in \mathfrak{q}.$ Finally, if $xy \in \mathfrak{q},$ again by (V3) it holds 
$\infty = v(xy) = v(x)+v(y),$ thus $x \in \mathfrak{q}$ or $y \in \mathfrak{q}.$
\end{proof}
\end{lemma}

\noindent
Any valuation $v$ on $R$ defines a quasi-order $\preceq$ on the set $R$ by declaring $x \preceq y$ if and only if $v(y) \leq v(x)$ for $x,y \in R.$ In analogy to the ordered case, the equivalence class $E_0$ of 
$0,$ with respect to $\sim \: := \sim_{\preceq},$ coincides with the support $\mathfrak{q}.$ \vspace{1mm}

\noindent
An immediate but important consequence of the axioms of a valuation above is the next result, which we will later use to get that any valued ring $(R,v)$ is a quasi-ordered ring.

\begin{lemma} \label{min} Let $v: R \to \Gamma \cup \{\infty\}$ be a valuation and $x,y \in R$ with $v(x) \neq v(y).$ Then $v(x+y) = \min\{v(x),v(y)\}.$
\begin{proof}
 As in the field case, see for instance \cite[p.20, (1.3.4)]{Prestel}. 
\end{proof}

\end{lemma}

\section{Quasi-ordered rings}
\noindent
Let $R$ denote a commutative ring with $1.$ \vspace{2mm}

\noindent
The notions of quasi-ordered fields and ordered rings (see the Definitions \ref{qofield} and \ref{oring}) suggest to call $(R,\preceq)$ a quasi-ordered ring, if $\preceq$ is a binary, reflexive, 
total and transitive relation on $R,$ such that the axioms (O1) -- (O3) and (Q3) are satisfied. However, as the following counterexample shows, this is not quite good enough to obtain the desired dichotomy. 
\vspace{2mm}

\noindent
Consider $R = \mathbb{R}[X,Y].$ We define a relation $\preceq'$ on the monomials in $R$ as follows: first declare $0 \prec' r$ for all $r \in R \backslash \{0\}.$ For $f = rX^iY^j$ and 
$g = sX^mY^n \in R\backslash\{0\}$ set

\[
 f \preceq' g : \Leftrightarrow \begin{cases} \textrm{ either } 0 < n \\
\textrm{ or } j=n=0 \textrm{ and } i \leq m
\end{cases}
\]

\noindent
So the relation $\preceq'$ can be described by the chain
\[
 0 \prec' 1 \prec' X \prec' X^2 \prec' X^3 \prec' \ldots \prec' Y \sim XY \sim X^2Y \sim \ldots \sim Y^2 \ldots
\]
and the rule $rf \sim f$ for all monomials $f$ and all $r \in \mathbb{R}\backslash\{0\}.$ Extend $\preceq'$ to the whole of $R$ by declaring for $0 \neq f,g \in R:$
\[
 f \preceq g :\Leftrightarrow \textrm{the} \preceq'\textrm{-largest monomial of } f \textrm{ is smaller or equivalent than the one of } g.
\]

\begin{proposition} \label{prop2} \hspace{7cm}
\begin{enumerate} 
\item $\preceq$ is a reflexive, transitive and total relation on $R$ such that $(R,\preceq)$ satisfies the axioms $\mathrm{(O1) - (O3)}$ and $\mathrm{(Q3)}.$
\item $\preceq$ is neither an order on $R,$ nor given by $f \preceq g \Leftrightarrow v(g) \leq v(f)$ for some valuation $v$ on $R.$
\end{enumerate}
\begin{proof} (1) Clearly $\preceq$ is reflexive, transitive, total and satisfies (O1). To prove (O2), note that $R$ is integer and that there are no $\preceq$-negative elements, so $fg \preceq 0$ yields $f = 0$ or 
$g=0.$ Hence, $f \preceq 0$ or $g \preceq 0.$ For (O3), simply observe that if $m$ is a $\preceq'$-larger monomial than $n,$ then $am$ is a $\preceq'$-larger monomial than $an$ for all $a \in R.$ It 
remains to verify axiom (QR3). So assume that $f \preceq g$ and $h \nsim g.$ The condition $h \nsim g$ particularly ensures, that the largest $\preceq'$-monomial of $g$ is not the additive inverse of the largest 
$\preceq'$-monomial of $h.$ Either the largest $\preceq'$-monomial of $h$ is strictly smaller than the one of $g$ and then $f+h \preceq g+h,$ or it is strictly greater and then $f+h \sim g+h.$ 
\vspace{1mm}

\noindent
(2) $\preceq$ is certainly not an order since $0 \prec -1.$ It is also not induced by some valuation on $R$ since $X\prec X^2$ and $0 \prec Y,$ but $XY \sim X^2Y.$ So if $\preceq$ would come from some valuation 
$v,$ then $2v(X) < v(X)$ and $v(Y) < \infty,$ but $2v(X) + v(Y) = v(X) + v(Y),$ a contradiction.
\end{proof}
\end{proposition}

\noindent
Thus, an additional axiom is required that rules this kind of counterexample out.

\begin{definition} \label{qoring} Let $R$ be a commutative ring with $1$ and $\preceq$ be a binary, reflexive, transitive and total relation on $R.$ Then $(R,\preceq)$ is called a \textbf{quasi-ordered ring} if 
$\forall x,y,z \in R:$
\begin{itemize}
\item[(QR1)] $0 \prec 1,$
\item[(QR2)] $xy \preceq 0 \Rightarrow x \preceq 0 \lor y \preceq 0,$
\item[(QR3)] $x \preceq y, \; 0 \preceq z \Rightarrow xz \preceq yz,$
\item[(QR4)] $x \preceq y, \; z \nsim y \Rightarrow x+z \preceq y+z,$
\item[(QR5)] If $0 \prec z,$ then $xz \preceq yz \Rightarrow x \preceq y.$
\end{itemize}
Moreover, we denote the equivalence class of $0$ (w.r.t. $\sim$) by $E_0.$
\end{definition}

\noindent
Axiom (QR5) prevents the counterexample from above, because from $0 \prec Y$ and $XY \sim X^2Y$ now follows $X \sim X^2,$ contradicting the definition of $\preceq.$

\begin{remark} \label{QR2} Note that (QR5) single-handedly implies (QR2). Indeed, if $xy \preceq 0$ and $x$ is strictly positive, then (QR5) yields that $y \preceq 0.$ However, we decided not to remove this axiom 
in order to preserve the analogy to Definition \ref{oring}. \\ Further note that our counterexample above proves that (QR5) is strictly stronger than (QR2), even when given all the other axioms from Definition 
\ref{qoring}.  
\end{remark}

\noindent
First of all we give two important classes as examples for quasi-ordered rings, thereby establishing one implication of the dichotomy.

\begin{example} \label{dichoeasy} \hspace{7cm}
 \begin{enumerate}
  \item Let $(R,\leq)$ be an ordered ring with support $\mathfrak{p}.$  Then $(R,\preceq)$ is a quasi-ordered ring with $E_0 = \mathfrak{p}.$
  \item Let $(R,v)$ be a valued ring with support $\mathfrak{q}.$ Then $x \preceq y :\Leftrightarrow v(y) \leq v(x)$ defines a quasi-order on $R$ with $E_0=  \mathfrak{q}.$
 \end{enumerate}
\begin{proof} \hspace{7cm}
\begin{enumerate}
 \item Comparing the definitions of ordered and quasi-ordered rings, we only have to check axiom (QR5). Note that if $y-x < 0,$ then $(y-x)z < 0$ (recall that $\mathfrak{p}$ is a prime ideal), so $yz < xz,$ a 
contradiction. Therefore $0 \leq y-x,$ i.e. $x \leq y.$
 \item Clearly $\preceq$ is reflexive, transitive and total, since the order $\leq$ on the value group of $v$ has these properties. Further note that $v(1) = 0 < \infty = v(0),$ thus $0 \prec 1.$ This shows
that (QR1) is fulfilled. For (QR2) there is nothing to show by the previous remark. Next we establish (QR3). From $x \preceq y$ follows $v(y) \leq v(x).$ Hence, $$v(yz) = v(y) + v(z) \leq v(x) + v(z) = v(xz),$$ and 
therefore $xz \preceq yz.$ The proof of (QR4) is done by case distinction. We do the case $y \not\preceq z,$ the case $z \not\preceq y$ being similar. So let $y \not\preceq z.$ Then $v(z) \nleq v(y),$ i.e. 
$v(y) < v(z).$ Moreover $v(y) \leq v(x).$ Applying Lemma \ref{min} yields $$v(y+z) = v(y) \leq \min\{v(x),v(z)\} \leq v(x+z),$$ hence $x+z \preceq y+z.$ Last but not least, we have to verify (QR5).  Note that 
$0 \prec a$ means $v(a) < \infty.$ Thus, $v(ay) \leq v(ax)$ implies $v(y) \leq v(x),$ whereby $ax \preceq ay$ gives us $x \preceq y.$ Finally $x \in E_0$ if and only if $v(x) = v(0) = \infty,$ i.e. if and only if 
$x \in \textrm{supp}(v).$
\end{enumerate}
\end{proof}
\end{example}

\begin{notation}
For the rest of this note let $(R,\preceq)$ always be a quasi-ordered ring. The previous example reveals that given an order or a valuation on $R,$ their support coincides with $E_0$ for some suitable quasi-order
on $R.$ Our investigation will particularly unfold that the converse is also true, i.e. $E_0$ is either the support of an order or else the support of a valuation on $R.$ Therefore we refer to $E_0$ as the 
\textbf{support} of the quasi-order $\preceq.$ 
\end{notation}

\noindent
We continue this section by showing that $E_0$ is a prime ideal (see Proposition \ref{prime}). From this we get immediately that our definition of quasi-ordered rings indeed generalizes Definition \ref{qofield}. 
Moreover, this result is crucial for our proof of the dichotomy, as we want to extend the quasi-order $\preceq$ on $R$ to the quotient field $\textrm{Quot}(R/E_0)$ (see Proposition \ref{quot}).

\begin{lemma} \label{lem1} If $x \nsim 0$ and $y \sim 0,$ then $x+y \sim x.$ Particularly, if $y \sim 0,$ then $-y \sim 0.$ 
\begin{proof} Since $x \nsim 0$ and $y \sim 0,$ it holds $x \nsim y.$ So from $y \preceq 0$ we obtain $x+y \preceq x$ and from $0 \preceq y$ we obtain $x \preceq x+y$ via Axiom (QR4). \\ Now if $y \sim 0,$ then 
$-y \nsim 0$ would mean that $0 = (-y)+y \sim -y \nsim 0,$ a contradiction.
\end{proof}
\end{lemma}

\noindent
The next result is a consequence of axiom (QR5). As a matter of fact it would have also been possible to take it as the additional axiom to prevent the counterexample from the beginning of this section, and 
then to deduce (QR5) from it. We will use both versions throughout this note.

\begin{lemma} \label{QR5+} Let $x,y,z \in R.$ If $z \nsim 0,$ then $xz \sim yz \Rightarrow x \sim y.$
\end{lemma} 
\begin{proof} For $0 \prec z,$ this is basically the same as (QR5). So suppose that $z \prec 0.$ By (QR4) and the previous lemma $0 \prec -z.$ Thus, (QR5) tells us $-x \sim -y.$ Assume for a contradiction that 
$x \nsim y,$ without loss of generality $x \prec y.$ By transitivity of $\preceq$ we get either $x \nsim -x,-y$ or $y \nsim -x,-y.$ If $y \nsim -x,-y,$ we obtain from $-x \preceq -y$ that $y-x \preceq 0.$ If 
$x \nsim 0,$ then (QR4) yields $y \preceq x.$ Otherwise, the same follows from Lemma \ref{lem1}. Hence, there is a contradiction anyway. So suppose that $x \nsim -x,-y.$ Then $-x \preceq -y$ implies 
$0 \preceq x-y,$ and $-y \preceq -x$ implies $x-y \preceq 0$ via (QR4). So $x-y \in E_0,$ but then also $y-x \in E_0$ by the previous lemma. Thus, $y-x \sim 0.$ From (QR4) (if $x \nsim 0),$ respectively Lemma 
\ref{lem1} (if $x \sim 0),$ we obtain $y \preceq x,$ again a contradiction.
\end{proof}



\begin{proposition} \label{prime} The support $E_0$ is a prime ideal of $R.$
\begin{proof} Let $x,y \in E_0.$ Assume for a contradiction that $x+y \notin E_0.$ So it holds $x+y \nsim 0,$ but $-y \sim 0$ by Lemma \ref{lem1}. The same lemma gives the contradiction
$0 \sim x = (x+y)-y \sim x+y \nsim 0.$ Now let $x \in E_0, r \in R.$ If $0 \preceq r,$ then $x \sim 0$ yields $xr \sim 0$ by (QR3), i.e. $xr \in E_0.$ If $r \prec 0,$ then $0 \prec -r.$ Moreover, $-x \sim 0$ 
(Lemma \ref{lem1}). Therefore $(-r)(-x) = rx \sim 0,$ so again $rx \in E_0.$ This proves that $E_0$ is an ideal, while axiom (QR1) states that $1 \notin E_0,$ i.e. $E_0 \neq R.$ Finally assume for a contradiction 
that $xy \in E_0,$ but $x,y \notin E_0.$ We may without loss of generality assume that $0 \prec x,$ for if $x,y \prec 0,$ then $0 \prec -x,-y.$ But if $0 \prec x$ and $xy \sim 0,$ then $y \sim 0$ by Lemma
\ref{QR5+}, a contradiction.
\end{proof}
\end{proposition}

\begin{corollary} \label{prop1} Suppose that $R$ is a field. Then $(R,\preceq)$ is a quasi-ordered field.
\begin{proof} Proposition \ref{prime} states that $E_0$ is a prime ideal, so $E_0 = \{0\}$ and (Q1) is fulfilled. The axioms for $+$ and $\cdot$ are the same in the ring and the field case.
\end{proof}
\end{corollary}

\section{The dichotomy}
\noindent
As before, let $(R,\preceq)$ denote a quasi-ordered ring with support $E_0.$ In the first part of this section we carry out how $\preceq$ can be extended to a quasi-order $\trianglelefteq$ on 
$K:=\textrm{Quot}(R/E_0).$ Applying Fakhruddin's dichotomy to $K,$ we know that $\trianglelefteq$ is either an order, or else comes from a valuation on $K.$ In the second part we deduce that then $\preceq$ is
also an order, respectively induced by a valuation on $R.$

\begin{lemma} \label{residue} $(R/E_0, \preceq')$ is a quasi-ordered ring satisfying $\overline{x} \sim 0 \Leftrightarrow \overline{x} = 0,$ where $\overline{x} \preceq' \overline{y} :\Leftrightarrow x \preceq y.$
\begin{proof} Well-definedness of $\preceq'$ is obtained from Lemma \ref{lem1} and Proposition \ref{prime} as follows: if $x \preceq y$ and $c,d \in E_0,$ these results state that $x+c \sim x \preceq y \sim y+d.$
It is clear that $\preceq'$ is reflexive, transitive and total. Also the axioms (QR1) -- (QR5) are easily verified. Finally, 
$\overline{x} \sim 0 \Leftrightarrow x \sim 0 \Leftrightarrow x \in E_0 \Leftrightarrow \overline{x} = 0.$
\end{proof}
\end{lemma}

\begin{lemma} It holds $0 \preceq x^2$ for all $x \in R.$
\begin{proof} Simply observe that $0 \preceq x$ or $0 \preceq -x$ for all $x \in R.$ Applying (QR3) results in $0 \preceq x^2.$
\end{proof}

\end{lemma}

\begin{proposition} \label{quot} Suppose that $(R, \preceq)$ satisfies $E_0 = \{0\}.$ Then $(K,\trianglelefteq)$ is a quasi-ordered field, where $K:=\mathrm{Quot}(R)$ and
$$\frac{a}{b} \trianglelefteq \frac{x}{y} :\Leftrightarrow aby^2 \preceq xyb^2$$
extends $\preceq$ from $R$ to $K.$
\end{proposition}  
\begin{proof} At first we make sure that $\trianglelefteq$ is well-defined, i.e. that $ay = bx$ and $\frac{a}{b} \trianglelefteq \frac{p}{q}$ yields $\frac{x}{y} \trianglelefteq \frac{p}{q}$ 
(the proof for $\trianglerighteq$ is the same). Thus, we have to show that $ay = bx$ and $abq^2 \preceq pqb^2$ means $xyq^2 \preceq pqy^2.$ From $abq^2 \preceq pqb^2$ follows $abq^2y^2 \preceq pqb^2y^2.$ Using 
$ay = bx,$ we obtain $b^2q^2xy \preceq pqb^2y^2.$ Axiom (QR5) gives us $q^2xy \preceq pqy^2.$ \\
Reflexivity, transitivity and totality of $\trianglelefteq$ are clear, as well as the fact that $\trianglelefteq$ extends $\preceq.$ Next we check transitivity. So we have to verify that 
$$\frac{a}{b} \trianglelefteq \frac{x}{y} \textrm{ and } \frac{x}{y} \trianglelefteq \frac{p}{q} \Rightarrow \frac{a}{b} \trianglelefteq \frac{p}{q},$$
i.e. that $aby^2 \preceq xyb^2 \textrm{ and } xyq^2 \preceq pqy^2$ implies $abq^2 \preceq pqb^2.$ \\
The two conditions imply that $aby^2q^2 \preceq xyb^2q^2$ and $xyq^2b^2 \preceq pqy^2b^2.$ Transitivity of $\preceq$ yields $aby^2q^2 \preceq pqy^2b^2$ and Axiom (QR5) indeed tells us $abq^2 \preceq pqb^2.$ \\
It remains to establish the axioms (Q1) -- (Q3).
\begin{itemize}
\item[(Q1)] Show $\frac{x}{y} \sim 0$ implies $\frac{x}{y} = 0.$ Indeed, if $\frac{x}{y} \sim \frac{0}{1},$ then $xy \sim 0,$ so $xy = 0.$ Since $R$ is integer and $y \neq 0,$ evidently $x = 0,$ and 
therewith  $\frac{x}{y} = 0.$
\item[(Q2)] Show $$\frac{a}{b} \trianglelefteq \frac{x}{y}, \ 0 \trianglelefteq \frac{p}{q} \Rightarrow \frac{ap}{bq} \trianglelefteq \frac{xp}{yq}.$$ 
From the two conditions follows $aby^2 \preceq xyb^2$ and $0 \preceq pq.$ The latter implies $0 \preceq pq^3.$ Multiplying $pq^3$ on both sides of the inequation $aby^2 \preceq xyb^2$ proves the claim.  
\item[(Q3)] Show $$\frac{a}{b} \trianglelefteq \frac{x}{y} \textrm{ and } \frac{p}{q} \nsim \frac{x}{y} \Rightarrow \frac{aq+bp}{bq} \trianglelefteq \frac{xq+yp}{yq}.$$
So let $aby^2 \preceq xyb^2$ and $pqy^2 \nsim xyq^2.$ We have to prove that
\[
(aq+bp)bqy^2q^2 \preceq (xq+yp)yqb^2q^2,
\]
or equivalently
\[
aq^4by^2 + b^2pq^3y^2 \preceq xq^4yb^2 + py^2q^3b^2.
\]
From $aby^2 \preceq xyb^2$ follows $aby^2q^4 \preceq xyb^2q^4.$ If we show $py^2q^3b^2 \nsim xq^4yb^2,$ then we are done by axiom (QR4). But this is an immediate consequence of the assumption $pqy^2 \nsim xyq^2$ and
the contraposition of Lemma \ref{QR5+} applied to $b^2q^2 \neq 0.$  
\end{itemize}
\end{proof}

\begin{lemma} \label{val} If $\overline{v}: R/E_0 \to \Gamma \cup \{\infty\}$ is a valuation with $\textrm{supp}(\overline{v}) = \{0\},$ then $v: R \to \Gamma \cup \{\infty\},$ $v(x) := v(\overline{x}),$ is a 
valuation on $R$ with $\textrm{supp}(v) = E_0.$
\end{lemma}
\begin{proof} Note that $v(x) = v(x+c)$ for all $c \in \textrm{supp}(v),$ because either $x$ is in the support of $v,$ and then so is $x+c,$ or $x$ is not and then we get the result by Lemma \ref{min}. It is 
clear that the value sets of $v$ and $\overline{v}$ are the same. Moreover, note that
\[
x \in \textrm{supp}(v) \Leftrightarrow v(x) = \infty \Leftrightarrow \overline{v}(\overline{x}) = \infty \Leftrightarrow \overline{x} = 0 \Leftrightarrow x \in E_0.
\]
The axioms of a valuation are easily verified.
\end{proof}

\begin{lemma} \label{ord} If $\preceq'$ is an order on $R/E_0$ with support $\{0\},$ then $x \preceq y :\Leftrightarrow \overline{x} \preceq' \overline{y}$ defines an order on $R$ with support $E_0.$ 
\begin{proof} 
Well-definedness follows precisely as in the proof of Lemma \ref{residue}. It holds $x \sim 0 \Leftrightarrow \overline{x} \sim 0 \Leftrightarrow \overline{x} = 0 \Leftrightarrow x \in E_0.$ The verification of the 
axioms of an order is trivial, as $\preceq'$ is an order by assumption.
\end{proof}
\end{lemma}

\noindent
We can now prove the following analogue of Theorem \ref{qofielddicho}.

\begin{theorem} \label{dicho} A quasi-ordered ring $(R,\preceq)$ is either an ordered ring or else a valued ring $(R,v)$ such that $x \preceq y$ if and only if $v(y) \leq v(x)$ for all $x,y \in R.$ Moreover, the 
support of the quasi-order coincides with the support of this order, respectively with the support of this valuation.
\begin{proof} Lemma \ref{residue} yields that $(R/E_0, \preceq')$ is a quasi-ordered ring with support $\{0\}.$ Let $K := \textrm{Quot}(R/E_0).$ By Proposition \ref{quot} we know that $(K,\trianglelefteq)$ is
a quasi-ordered field such that $\trianglelefteq$ extends $\preceq'.$ So by exploiting Fakhruddin's dichotomy (Theorem \ref{qofielddicho}), we obtain that $\trianglelefteq$ either comes from a valuation on $K$ or 
else is an order on $K.$ Either way, the corresponding support is trivial by definition. \\
In the first case there exists a valuation $\overline{v'}: K \to \Gamma \cup \{\infty\}$ such that for all $\overline{x},\overline{y} \in K$ holds
$\overline{x} \trianglelefteq \overline{y} \Leftrightarrow \overline{v'}(\overline{y}) \leq \overline{v'}(\overline{x}).$
The restriction $\overline{v}$ of $\overline{v'}$ to $R/E_0$ is obviously a valuation on $R/E_0$ with support $\{0\}$ and the same value group such that
$ \overline{x} \preceq' \overline{y} \Leftrightarrow \overline{v}(\overline{y}) \leq \overline{v}(\overline{x})$ for all $\overline{x},\overline{y} \in R/E_0.$ Lemma \ref{val} implies that 
$v: R \to \Gamma \cup \{\infty\},$ $v(x) = \overline{v}(\overline{x}),$ is a valuation on $R$ with support $E_0.$ Moreover, by Lemma \ref{residue} and Lemma \ref{val},
\[
\forall x,y \in R: v(y) \leq v(x) \Leftrightarrow \overline{v}(\overline{y}) \leq \overline{v}(\overline{x}) \Leftrightarrow \overline{x} \preceq' \overline{y} \Leftrightarrow x \preceq y.
\]
If $\trianglelefteq$ is an order extending $\preceq',$ then $(R/E_0,\preceq')$ is already an ordered ring with support $\{0\}.$ Applying Lemma \ref{ord} yields that $(R,\preceq)$ is also an ordered ring with
support $E_0.$
\end{proof}
\end{theorem} 

\noindent
Incorporating Example \ref{dichoeasy}, we can formulate the previous theorem also as follows:

\begin{theorem} \label{main}
Let $R$ be a commutative ring with $1$ and $\preceq$ a binary relation on $R.$ Then $(R, \preceq)$ is a quasi-ordered ring if and only if it is either an ordered ring or else there is a valuation $v$ on $R$ such 
that $x \preceq y \Leftrightarrow v(y) \leq v(x).$ Moreover, the support of the quasi-order coincides with the support of the order, respectively with the support of the valuation.
\end{theorem} \vspace{1mm}

\section*{Acknowledgement}
\noindent
The results presented here are part of my PhD. I am grateful to my supervisor Salma Kuhlmann for her dedicated support in general, and for carefully reading this note in particular. Her numerous comments led to a
significant improvement. I would also like to thank Tom-Lukas Kriel for helpful discussions on the subject. 

\vspace{6mm}

\textsc{Fachbereich Mathematik und Statistik, Universit\"at Konstanz},

78457 \textsc{Konstanz, Germany}, 

E-mail address: simon.2.mueller@uni-konstanz.de
\end{document}